\documentclass[en, secnum, bibtex]{./cls/myamspaper}

\usepackage[top=40mm, bottom=37mm, left=30mm, right=30mm]{geometry}

\bibliography{./bib/bibliography.bib}

\title[On equicontinuous factors of flows on locally path-connected spaces]{On equicontinuous 
factors of flows on locally path-connected compact spaces}

\DeclareMathOperator{\irr}{irr}
\newcommand{\mon}{\mathrm{mon}}
\newcommand{\eq}{\mathrm{eq}}

\newcommand{\noop}[1]{}

\begin{document}

\begin{abstract}
We consider a locally path-connected compact metric space $K$ with finite 
first Betti number $\ub_1(K)$ and a flow $(K, G)$ on $K$ such that $G$ is abelian and
all $G$-invariant functions $f\in\uC(K)$ are constant. We prove that every equicontinuous 
factor of the flow $(K, G)$ is isomorphic to a flow on a compact abelian Lie group of 
dimension less than $\nicefrac{\ub_1(K)}{\ub_0(K)}$.
For this purpose, we use and provide a new proof for \cite[Theorem 2.12]{HauserJaeger2017} which states
that for a flow on a locally connected 
compact space the quotient map onto the maximal equicontinuous factor is \emph{monotone}, i.e., 
has connected fibers. Our alternative proof is a simple consequence of a new characterization of the 
monotonicity of a quotient map $p\colon K\to L$ between locally connected compact spaces $K$ and $L$ 
that we obtain by characterizing the local connectedness of $K$ in terms of the Banach lattice 
$\uC(K)$.
\end{abstract}

\subjclass[2010]{54H20, 37B05}

\maketitle

The study of topological dynamical systems via their maximal equicontinuous 
factors plays an important role for, e.g., tiling dynamical systems (see \cite[Chapter 5]{MathOfAperiodicOrder}),
Toeplitz flows (see \cite{Downarowicz2005}),
or the Furstenberg structure theorem for minimal distal flows. One reason is that,
for group actions, the maximal 
equicontinuous factor coincides with the Kronecker 
factor. The latter is highly structured, is minimal if and 
only if it is isomorphic to a minimal rotation on a homogeneous space of some compact group,
and also captures spectral-theoretic information.
In light of this, it is important to understand how the specific structure and 
properties of a system can be used to determine its maximal equicontinuous factor.
For example, it is known that if $(M, G)$ is a distal minimal flow on a compact 
manifold $M$, then its maximal equicontinuous factor is a flow on a homogeneous space of some 
compact Lie group, see \cite[Theorem 1.2]{Rees1977} or \cite[Theorem 1.2]{Ihrig1984}. 
If, additionally, the acting group $G$ is abelian, the maximal equicontinuous factor 
is in fact isomorphic to a flow on a compact abelian Lie group.
For non-distal systems, however, few results in this spirit seem to exist. Notably, Hauser and J\"ager recently proved the following.

\begin{theorem*}[\mbox{\cite[Theorem 3.1]{HauserJaeger2017}}]
  Suppose that $f$ is a homeomorphism of the two-torus. If the maximal equicontinuous 
  factor of $(\T^2, f)$ is minimal, then it must be one of the following three:
  \begin{enumerate}[(i)]
    \item an irrational translation on the two-torus,
    \item an irrational rotation on the circle,
    \item the identity on a singleton.
  \end{enumerate}
\end{theorem*}

Thus, the geometric properties of the two-torus imply that the maximal 
equicontinuous factor of a flow on it must have a relatively simple structure 
if it is minimal: It is a rotation 
on a compact abelian Lie group of dimension less than two. As it turns out, this 
is representative of the following general phenomenon which is the main result 
of this article. Recall that every compact abelian 
Lie group is isomorphic to the product $F\times\T^m$ of a finite abelian group $F$ and a torus $\T^m$.

\begin{theorem*}
  Let $(K, G)$ be a flow such that 
  $K$ is locally path-connected with finite first Betti number $\ub_1(K)$, 
  $G$ is abelian, and such that $K$ is metrizable or $G$ is separable.
  If all $G$-invariant functions $f\in\uC(K)$ are constant, then every 
  equicontinuous factor of $(K, G)$ is isomorphic to a minimal flow on some compact abelian 
  Lie group of dimension less than $\nicefrac{\ub_1(K)}{\ub_0(K)}$. More precisely,
  for every 
  equicontinuous factor $(L, G)$ of $(K, G)$, there are
  a finite abelian group $F$ of order $|F| \leq \ub_0(K)$ and an
  $m \leq \nicefrac{\ub_1(K)}{\ub_0(K)}$ such 
  that $(L, G)$
  is isomorphic to a minimal action of $G$ on $F\times\T^m$ via rotations.
\end{theorem*}

This provides a bound on the complexity of the maximal equicontinuous factor 
in terms of topological invariants of the underlying space $K$ and applies in 
particular to minimal systems on compact manifolds.
As a corollary, we obtain in \cref{cor:liegroupquotient} and \cref{cor:toruscase} that the above-cited 
\cite[Theorem 3.1]{HauserJaeger2017} holds analogously for tori of arbitrary dimension
and, in an appropriate version, more generally for quotients $H/\Gamma$ of connected,
simply connected Lie groups $H$ by discrete, cocompact subgroups $\Gamma \subset H$. Examples 
for such spaces are in particular given by nilmanifolds, see \cite[Chapter 10]{HostKra2018}.

Another result of \cite{HauserJaeger2017} and a key element in their proof of 
\cite[Theorem 3.1]{HauserJaeger2017} is that for a large 
class of flows the factor map onto the maximal equicontinuous factor is \emph{monotone},
meaning that preimages of points are connected. 
Similar results were, it seems, first obtained in  \cite{McMahonWu1976} where the 
authors proved that for each extension $p\colon (K, G) \to (L, G)$ of minimal flows 
that decomposes into a tower of equicontinuous extensions, the quotient map $K \to K/\uS(p)$ 
is monotone, where $\uS(p)$ denotes the relativized equicontinuous structure relation.
Hence, for a distal minimal flow its Furstenberg tower consists entirely of 
monotone quotient maps, see \cite[Proposition 2.2]{Greschonig2014}, and in particular 
the map onto the maximal equicontinuous factor is monotone. Without such structural assumptions,
the monotonicity of the maximal equicontinuous factor does not hold in general (take, e.g., the extension of the 
shift $\tau\colon \Z \to \Z, x\mapsto x+1$ to the one-point compactification of $\Z$). 
If, however, the underlying space is locally connected, it is shown in \cite[Theorem 2.12]{HauserJaeger2017} 
that the quotient map onto the maximal 
equicontinuous factor is indeed monotone which is notable since monotone quotient maps relate the geometry of a
space to that of a quotient: Every monotone quotient map between suitable spaces induces a
surjective homomorphism on the level of fundamental groups, see \cite[Theorem 1.1]{Calcut2012}
(stated below as \cref{thm:calcut}).
Since this idea will be crucial for the proof of \cref{mthm}, we also provide an alternative 
proof for the monotonicity of the maximal equicontinuous factor under the 
assumption of local connectedness. This is based on two results that are of interest by themselves:
We characterize the local connectedness of a compact space $K$ in terms of the Banach lattice 
$\uC(K)$ and then use this to give a new characterization for the monotonicity of a 
quotient map between locally connected compact spaces. (For background information on 
Banach lattices, see \cite[Section 7.1]{EFHN2015}.) The above-mentioned monotonicity result then is 
a simple application of these characterizations. We prove these results in
\cref{section:monotonicity} after collecting some preliminaries in \cref{section:prelim}.
The main result is proved in \cref{section:mef}.

\textbf{Notation and Terminology.} By a \emph{topological dynamical system} $(K, S, \Phi)$
we mean a continuous action $\Phi\colon S\times K \to K$ of a topological semigroup $S$ on a 
compact space $K$. (We always assume compact spaces to be Hausdorff.) We usually drop $\Phi$ from 
the notation and write $sx$ instead of $\Phi(s, x)$ for $s\in S$ and $x\in K$. 
We simply call $(K, S)$ a \emph{flow} if $S$ is a 
group. If we refer to a pair $(K, \phi)$ of a compact space $K$ and a continuous map $\phi\colon K\to K$ 
as a topological dynamical system, we regard the $\N$ or $\Z$-action on $K$ given by the powers of $\phi$, 
depending on whether $\phi$ is explicitly specified to be invertible or not. By an 
\emph{extension} $p\colon (K, S)\to (L, S)$ of topological dynamical systems we mean a continuous, 
surjective, $S$-equivariant function 
$p\colon K\to L$. Given such an extension, we also call $p$ a \emph{factor map} and $(L, S)$ a \emph{factor} of $(K, S)$.
A system $(K, S)$ is called \emph{equicontinuous} if the family $\{ \Phi(s, \cdot) \mid s\in S\} 
\subset \uC(K, K)$ is equicontinuous.

If $\phi\colon K\to L$ is a continuous function between compact spaces, we
denote by $T_\phi$ its \emph{Koopman operator} 
\begin{align*}
  T_\phi\colon\uC(L)&\to \uC(K), \\
  f&\mapsto f\circ \phi.
\end{align*}
We assume the reader to be familiar with Koopman operators 
and the theory of commutative $\uC^*$-algebras and refer to \cite[Chapter 4]{EFHN2015} for 
background information.
If $H$ is an abelian group, then $\rank(H)$ denotes its torsion-free rank, i.e., 
the dimension of the $\Q$-vector space $\Q\otimes H$. For a topological space $X$ and $i\in\N_0$, 
we denote by $\ub_i(X)$ the $i$-th Betti number of $X$. Note that for a compact, locally 
path-connected space $X$, $\ub_0(X)$ is the (finite) number of connected components of $X$.

\textbf{Acknowledgements.} Part of this work was 
done during a stay at Kiel University and the author is very grateful to M.\ Haase 
for his kind hospitality during this time, as well as for bringing the book \cite{Bronstein1979}
and thereby \cite{Rees1977} and \cite{Ihrig1984} to the author's attention.

\section{Factors and invariant subalgebras}\label{section:prelim}

Consider the categories $\mathbf{CTop}$ of 
compact topological spaces and $\mathbf{C^*_{com,1}}$ of commutative unital $\mathrm{C}^*$-algebras 
as well as the contravariant functor 
\begin{align*}
  \mathrm{C}\colon \mathbf{CTop} \to \mathbf{C^*_{com, 1}}
\end{align*}
given by $K\mapsto \uC(K)$ and $\phi \mapsto T_\phi$ for compact 
spaces $K$ and $L$ and continuous functions $\phi\colon K\to L$. It 
is a consequence of the classical Gelfand representation theorem (see \cite[Theorem 4.23]{EFHN2015})
that $\uC$ is an antiequivalence of categories. This
allows to use operator-theoretic concepts to understand topological dynamical 
systems (cf.\ \cref{thm:mefequi}) but also, conversely, to use 
geometric tools to obtain results about operators (cf.\ \cref{cor_spectrum}).
A particular consequence of this antiequivalence that we will use throughout the article 
is the relationship between factors of a topological dynamical system $(K, S)$ and 
$S$-invariant unital $\uC^*$-subalgebras of $\uC(K)$: Suppose $p\colon (K, S) \to (L, S)$
is a factor map of topological dynamical systems. Then the Koopman operator 
\begin{align*}
  T_p \colon \uC(L)&\to \uC(K) \\
  f &\mapsto f\circ p
\end{align*}
is an $S$-equivariant $\uC^*$-embedding and its image $\mathcal{A}_L \defeq T_p(\uC(L)) \subset \uC(K)$ is an $S$-invariant 
unital $\uC^*$-subalgebra of $\uC(K)$, i.e., $T_s(\mathcal{A}_L) \subset \mathcal{A}_L$ for each $s\in S$. 
Moreover, if $q\colon (K, S) \to (M, S)$
is another factor map such that $T_q(\uC(M)) = T_p(\uC(L))$, then 
\begin{align*}
   \Psi \defeq T_p^{-1}\circ T_q \colon \uC(M) \to \uC(L)
\end{align*}
defines an $S$-equivariant $\uC^*$-isomorphism and so there is a unique $S$-equivariant homeomorphism $\eta \colon L\to M$ 
such that $\Psi = T_\eta$, making the following diagram commutative:
\begin{align*}
  \xymatrix{
    & (K, S) \ar[dl]_-p \ar[dr]^-q & \\
    (L, S) \ar[rr]^-\eta & & (M, S)
  }
\end{align*}
This shows that a factor of $(K, S)$ is, up to isomorphy, 
uniquely determined by its corresponding $\uC^*$-subalgebra of $\uC(K)$.
Conversely, every $S$-invariant unital $\uC^*$-subalgebra of $\uC(K)$
canonically corresponds to a factor of $(K, S)$ via the Gelfand representation theorem 
and one thereby obtains, again up to isomorphy, a one-to-one correspondence 
between factors of $(K, S)$ and $S$-invariant unital $\uC^*$-subalgebras
of $\uC(K)$. For the convenience of the reader, \cref{rem:algebra-quotient} below 
explains a more elementary approach to the this correspondence.

Given that the $\uC^*$-subalgebras of many factors such as the maximal trivial factor, 
the maximal equicontinuous 
factor, the maximal tame factor, the Kronecker factor, or the Abramov factor admit 
relatively simple descriptions, it is convenient to study factors and factor
maps via their corresponding subalgebras. We do so in \cref{cor_algebrachar} to give a 
simple criterion characterizing the monotonicity of a factor via its corresponding 
$\uC^*$-subalgebra. We will then see that this criterion is readily verified for the 
$\uC^*$-subalgebra of the maximal equicontinuous factor.

\begin{example}\label{example:factor}
  Let $(K, S)$ be a topological dynamical system.
  \begin{enumerate}
    \item\label{example:factor_a} The factor consisting of a single point corresponds 
    to the subalgebra $\C \1_K \subset \uC(K)$ of constant functions.
    \item\label{example:factor_b} Let $p\colon (K, S) \to (L, S)$ be a factor map onto 
    a \emph{trivial} factor, i.e., one on which $S$ acts trivially. Then $T_p(\uC(L))$ 
    is a 
    subalgebra of  
    \begin{align*}
      \mathcal{A}_{\mathrm{triv}} \defeq \left\{ f\in \uC(K) \mmid \forall s\in S\colon T_s f = f \right\}.
    \end{align*}
    This \enquote{fixed} algebra
    corresponds to the \emph{maximal trivial factor} $(K_\mathrm{triv}, S)$ of $(K, S)$
    through which every factor map onto another trivial factor factorizes.
    \item\label{example:factor_c} Similarly, the subalgebra
    \begin{align*}
      \mathcal{A}_{\mathrm{eq}} \defeq \left\{f\in \uC(K) \mmid \left\{T_s f \mmid s\in S \right\} \text{is equicontinuous}\right\}
    \end{align*}
    corresponds to the \emph{maximal equicontinuous factor} $(K_\eq, S)$ of $(K, S)$ since a 
    topological dynamical system $(M, S)$ is equicontinuous if and only if 
    for each $g\in\uC(M)$ the orbit $\{T_s g \mid s\in S\}$ is equicontinuous. By 
    the Arzel\`a-Ascoli theorem, this is equivalent to the orbits $\{T_s g \mid s\in S\}$ being 
    relatively compact for each $g\in\uC(M)$. If $S$ acts via homeomorphisms, this means that 
    the maximal equicontinuous factor $(K_\eq, S)$ coincides with the \emph{Kronecker factor}
    which corresponds, for abelian $S$, to the $\uC^*$-subalgebra
    \begin{align*}
      \mathcal{A} \defeq \overline{\lin}\left\{f\in\uC(K) \mid \forall s\in S \exists \lambda_s \in \T \colon T_sf = \lambda_s f\right\}
    \end{align*}
    spanned by the eigenfunctions of the action of $S$, see \cite[Corollary 16.32]{EFHN2015}.
  \end{enumerate}
\end{example}

\begin{remark}\label{rem:mtf}
  Let $(M, G)$ be an equicontinuous flow. Then each orbit closure is a minimal subset 
  of $M$ (see \cite[Lemma 2.3]{Auslander1988}) and so $M$ decomposes into minimal subsets.
  Moreover, the orbit closure relation
  \begin{align*}
    {\sim_G} \defeq \left\{(x,y)\in M\times M \mmid \overline{Gx} = \overline{Gy} \right\}.
  \end{align*}
  is a closed equivalence relation (see \cite[Exercise 2.6]{Auslander1988}) and a moment's 
  though reveals that hence, $M/\sim_G$ together with the trivial $G$-action is the maximal 
  trivial factor of $(M, G)$. In particular, $(M, G)$ is minimal if and only if 
  the maximal trivial factor of $(M, G)$ is a point. We note for the proof of \cref{mthm} that 
  as a consequence,
  given an arbitrary system $(K, G)$, its maximal equicontinuous factor is minimal 
  if and only if every $G$-invariant function $f\in\uC(K)$ is constant.
\end{remark}

\begin{remark}\label{rem:algebra-quotient}
  Given a compact space $K$, one can describe the relationship between compact 
  quotients $L$ of $K$ and unital $\uC^*$-subalgebras $\mathcal{A}$ of $\uC(K)$ without using
  the Gelfand representation theory: If $p\colon K\to L$ is a continuous surjective map
  onto a compact space $L$, define the unital $\uC^*$-subalgebra 
  \begin{align*}
    \mathcal{A}_L \defeq \{f\in \uC(K) \mid \forall l\in L \forall x, y\in p^{-1}(l)\colon f(x) = f(y)\}
  \end{align*}
  of functions constant on fibers of $p$ and note that $\mathcal{A}_L = T_p(\uC(L))$. Conversely,
  if $\mathcal{A}\subset \uC(K)$ is a unital $\uC^*$-subalgebra, define the closed equivalence 
  relation
  \begin{align*}
    {\sim_{\mathcal{A}}} \defeq \{ (x, y)\in K\times K \mid \forall f\in\mathcal{A}\colon f(x) = f(y)\}.
  \end{align*}
  and set $L_\mathcal{A} \defeq K/{\sim_{\mathcal{A}}}$.
  Then the assignments $L \mapsto \mathcal{A}_L$ and $\mathcal{A} \mapsto L_\mathcal{A}$ are,
  up to isomorphy of the compact spaces, mutually inverse. Analogously, one obtains 
  the above-explained correspondence between factors and invariant subalgebras if one 
  considers topological dynamical systems $(K, S)$ on $K$.
\end{remark}

\section{Local connectedness and monotonicity of factors}\label{section:monotonicity}

As noted in the introduction, one cannot expect the factor map onto the maximal 
equicontinuous factor of a flow to be monotone in general, i.e., its preimages 
of points are not generally connected. Therefore,
we focus on
quotient maps $p\colon K\to L$ between locally connected compact spaces and characterize 
their monotonicity in terms of the subalgebra $\mathcal{A}_L \subset \uC(K)$ of functions 
constant on fibers of $p$. We then apply this to the maximal equicontinuous factor. 
Recall the following elementary results on locally connected spaces.

\begin{lemma}\label{lem:locallyconnected}
  Let $X$ and $Y$ be topological spaces. 
  \begin{enumerate}
    \item\label{lem:locallyconnected_a} $X$ is locally connected if and only if for every open set $O\subset X$
    each connected component of $O$ is open in $X$.
    \item\label{lem:locallyconnected_b} $X$ is locally connected if and only if for every
    basis $\mathcal{B}$ for the topology of $X$ and every $O\in\mathcal{B}$ each 
    connected component of $O$ is open in $X$.
    \item\label{lem:locallyconnected_c} If $X$ is compact, $X$ is locally connected if and only if it is uniformly locally connected,
    i.e., if each entourage $U\in\mathcal{U}_X$ contains an entourage $V\in\mathcal{U}_X$ 
    such that $V[x]$ is connected for each $x\in X$.
    \item\label{lem:locallyconnected_d} If $X$ is locally connected and $f\colon X\to Y$ is a surjective quotient map, 
    then $Y$ is locally connected.
    \item\label{lem:locallyconnected_e} If $X$ is compact and locally connected, $X$ has only finitely many connected 
    components.
  \end{enumerate}
\end{lemma}
\begin{proof}
  For \ref{lem:locallyconnected_a},
  \ref{lem:locallyconnected_d}, and \ref{lem:locallyconnected_e} see Theorem 27.9, Corollary 27.11, and 
  Theorem 27.12 of \cite{Willard2004} and for \ref{lem:locallyconnected_c} see \cite[Proposition 9.39]{James1999}.
  \Cref{lem:locallyconnected_b} follows from the definition of 
  local connectedness and \cref{lem:locallyconnected_a}.
\end{proof}

Since in a compact space $K$ the sets of the form 
$[f\neq 0]$ for continuous, complex-valued functions $f\in\uC(K)$ constitute 
a base of the topology, \cref{lem:locallyconnected}\ref{lem:locallyconnected_a} and 
\ref{lem:locallyconnected_b} provide
a natural way to characterize the local connectedness of $K$ purely in terms of the Banach lattice $\uC(K)$.
For this purpose, we call $f, g\in \uC(K)$ \emph{orthogonal}
and write $f\perp g$ if they are orthogonal in the Banach lattice $\uC(K)$. (That is,
$f\perp g$ if and only if $|f|\wedge |g| = 0$ which is equivalent to $fg = 0$.) A decomposition 
$f = g + h$ is called \emph{orthogonal} if $g$ and $h$ are orthogonal. A function $f\in\uC(K)$ 
is called \emph{reducible} if $f = 0$ or if there is an orthogonal decomposition $f = g + h$ with 
nonzero $g, h\in\uC(K)$ and $f$ is called \emph{irreducible} otherwise. If $f = g + h$ is an 
orthogonal decomposition and $g$ is irreducible, then $g$ is 
called an \emph{irreducible part} of $f$. For a function $f\in\uC(K)$, define
\begin{align*}
  \irr(f) \defeq \{g\in \uC(K) \mid g \text{ is an irreducible part of } f\}
\end{align*}
and for a subset $\mathcal{F}\subset \uC(K)$, set
\begin{align*}
  \irr(\mathcal{F}) \defeq \bigcup_{f\in \mathcal{F}} \irr(f).
\end{align*}
A decomposition 
$f = \sum_{g\in \mathcal{F}} f$ for some at most countable set $\mathcal{F}\subset \uC(K)$ 
is called \emph{irreducible} if all $g\in\mathcal{F}$ are irreducible and pairwise orthogonal
and the sum converges uniformly to $f$.

\begin{proposition}\label{lc_char}
  Let $K$ be a compact space.
  \begin{enumerate}
    \item\label{lc_char_a} If $f\in \uC(K)$ and $M\subset [f\neq 0]$, then $\1_M f \in\uC(K)$
    if and only if $M$ is clopen in $[f\neq 0]$.
    \item\label{lc_char_b} If $f\in\uC(K)$, then $f$ is irreducible if and only if $[f\neq 0]$ is connected.
    \item\label{lc_char_c} For each $f\in\uC(K)$
    \begin{align*}
      \irr(f) = \left\{ \1_O f \mid O \text{ is an open connected component of } [f \neq 0] \right\}.
    \end{align*}
    \item\label{lc_char_d} For each $f\in \uC(K)$ and $\epsilon > 0$ the set 
    \begin{align*}
      \left\{ g\in \irr(f) \mid \|g\|_\infty > \epsilon \right\}
    \end{align*}
    is finite. In particular, $\irr(f)$ is countable.
    \item\label{lc_char_e} $K$ is locally connected if and only if each $f\in\uC(K)$ admits a unique 
    irreducible decomposition. In that case, the irreducible decomposition is given by
    \begin{align*}
      f = \sum_{g\in\irr(f)} g.
    \end{align*}

  \end{enumerate}
\end{proposition}
\begin{proof}
  For a fixed $f\in\uC(K)$, the multiplication operator 
  \begin{align*}
    \uC_\ub( [ f \neq 0] ) \rightarrow \uC(K), \quad g \mapsto gf
  \end{align*}
  is well-defined and if $M\subset [ f\neq 0 ]$ is clopen, then 
  $\1_M|_{[f\neq 0]} \in \uC_\ub( [f\neq 0] )$. Therefore, $f\1_M \in \uC(K)$. Conversely, if $M\subset [f\neq 0]$
  is such that $\1_M f \in \uC(K)$, the restriction
  $\1_M f|_{[f\neq 0]}$ is continuous and so dividing by $f|_{[f\neq 0]}$ yields the continuity 
  of $\1_M|_{[f\neq 0]}$.
  This proves \ref{lc_char_a} which in turn yields \ref{lc_char_b}.
  
  If $O$ is an open connected component of $[f\neq 0]$, then $\1_O f \in\irr(f)$ by 
  \ref{lc_char_a} and \ref{lc_char_b}. Conversely, take $g\in \irr(f)$. Then 
  $g = \1_{[g\neq 0]}f$ and so by \ref{lc_char_a}, $[g\neq 0]$ is a clopen subset of
  $[f\neq 0]$ and hence a union of connected components of $[f\neq 0]$. However, since 
  $g$ is irreducible, $[g\neq 0]$ is connected and so it is
  an open connected component of $[f\neq 0]$, proving \ref{lc_char_c}.
  Moreover, \ref{lc_char_c} yields that $\irr(f)$ is a bounded
  and equicontinuous set in $\uC(K)$ and so by the Arzel\`a-Ascoli theorem, $\irr(f)$ is relatively
  compact. Since by \ref{lc_char_c}, for every two $g, h\in\irr(f)$ with $g\neq h$ one 
  has $\|g-h\| = \max\{\|g\|, \|h\|\}$, \ref{lc_char_d} follows from the relative compactness
  of $\irr(f)$.

  Now suppose $K$ to be locally connected and take $f\in\uC(K)$. Since $K$ 
  is locally connected, each connected component of $[f\neq 0]$ is open 
  and so by \ref{lc_char_c} and \ref{lc_char_d}, the sum 
  $\sum_{g\in\irr(f)} g$
  converges uniformly to $f$. Hence, $f$ admits an irreducible decomposition which is 
  readily verified to be unique. Conversely, assume that each $f\in\uC(K)$ admits 
  an irreducible decomposition and let $x\in K$. To show that $K$ is locally
  connected at $x$, let $U\subset K$ be an open 
  neighborhood of $x$. Since $K$ is completely regular, there exists an
  $f\in \uC(K)$ with $x \in [f\neq 0] \subset U$. Moreover, since $f$ admits a unique
  irreducible decomposition, there is a unique $g\in \irr(f)$ such that $g(x) = f(x) \neq 0$.
  In particular, $x\in [g \neq 0] \subset [f \neq 0] \subset U$ and since $g$ is irreducible, $[g\neq 0]$ is connected,
  showing that $K$ is locally connected.
\end{proof}

After these preparatory notes on local connectedness, we now turn towards the notion of monotonicity and 
its characterizations. We restrict to compact spaces although many of the arguments are easily 
adapted to completely regular spaces.

\begin{definition}
  Let $X$ and $Y$ be topological spaces and $p\colon X\to Y$ a
  map. Then $p$ is called \emph{monotone} if for each 
  $y\in Y$ the preimage $p^{-1}(y)$ is a connected subset of $X$.
\end{definition}

If $p\colon K\to L$ is a continuous surjective map, then as noted in \cref{section:prelim} and in particular 
\cref{rem:algebra-quotient}, $L$ is, up to isomorphy, 
uniquely determined by the subalgebra $\mathcal{A}_L = T_p(\uC(L)) \subset \uC(K)$ of functions 
constant on the fibers of $p$. If $K$ is locally connected, this 
allows to use \cref{lc_char} to characterize the monotonicity of a quotient map $p\colon K\to L$
in terms of the subalgebra $\mathcal{A}_L$ and the Koopman operator $T_p$.

\begin{proposition}\label{prop:monotonechar}
  Let $K$ and $L$ be compact spaces, $K$ locally connected, and $p\colon K\to L$
  continuous and surjective. Then the following assertions are equivalent.
  \begin{enumerate}
    \item\label{part:monotonechar_a} $p$ is monotone.
    \item\label{part:monotonechar_b} For every connected set $C \subset L$ the 
    preimage $p^{-1}(C)$ is connected.
    \item\label{part:monotonechar_c} For every open, connected set $U\subset L$ the 
    preimage $p^{-1}(U)$ is connected.
    \item\label{part:monotonechar_d} For every irreducible $f\in\uC(L)$ the set $p^{-1}([f \neq 0])$ is connected.
    \item\label{part:monotonechar_e} $T_p$ preserves irreducibility of functions.
    \item\label{part:monotonechar_f} The subalgebra $\mathcal{A}_L = T_p(\uC(L)) \subset \uC(K)$
    satisfies $\irr(\mathcal{A}_L) \subset \mathcal{A}_L$.
  \end{enumerate}
\end{proposition}
\begin{proof}
  For the implication \ref{part:monotonechar_a} $\implies$ \ref{part:monotonechar_b},
  suppose $C\subset L$ to be connected and that $p^{-1}(C) = U\cup V$ for disjoint, open sets 
  $U, V\subset p^{-1}(C)$. Then $U$ and $V$ are saturated, i.e., $p^{-1}(p(U)) = U$ and 
  $p^{-1}(p(V)) = V$, since each 
  fiber of $p$ over $C$ is connected.
  Hence, the open (!) sets $p(U)$ and $p(V)$ form a cover of $C$ by 
  disoint, open sets. Since $C$ is connected, $U = \emptyset$ or 
  $V = \emptyset$ and so $p^{-1}(C)$ is connected.
  
  The implication \ref{part:monotonechar_b} $\implies$ \ref{part:monotonechar_c} 
  is trivial. For the implication \ref{part:monotonechar_c} $\implies$ \ref{part:monotonechar_a},
  note that for $l \in L$
  \begin{align*}
    p^{-1}(l) 
    = \bigcap_{U\in\mathcal{U}(l)} p^{-1}(U)
    = \bigcap_{\substack{U\in\mathcal{U}(l) \\ \text{closed, connected}}} p^{-1}(U)
  \end{align*}
  as $L$ is locally connected by \cref{lem:locallyconnected}. Since, in a compact space, the 
  intersection of a decreasing family of closed, connected subsets is again connected, 
  $p$ is monotone.
  
  The equivalence of \ref{part:monotonechar_d} and \ref{part:monotonechar_e} 
  follows since $[T_p(f) \neq 0] = [f\circ p \neq 0] = p^{-1}([f\neq 0])$.
  Moreover, \ref{part:monotonechar_e} and \ref{part:monotonechar_f} are 
  seen to be equivalent using the existence of irreducible decompositions for 
  functions in $\uC(L)$. Finally, 
  the implication \ref{part:monotonechar_c} $\implies$ \ref{part:monotonechar_d}
  is trivial and the converse implication follows analogously to 
  \ref{part:monotonechar_c} $\implies$ \ref{part:monotonechar_a} because, $L$ being locally 
  connected and completely regular, the sets of the form $[f \neq 0]$ for irreducible 
  $f\in\uC(L)$ form a basis of the topology of $L$ which allows to copy the argument.
\end{proof}

\begin{definition}
  Let $K$ be a compact space. A unital $\uC^*$-subalgebra $\mathcal{A}\subset \uC(K)$
  is called \emph{monotone} if the canonical quotient map $K \to K/\sim_{\mathcal{A}}$ is
  monotone.
\end{definition}

\begin{corollary}\label{cor_algebrachar}
  Let $K$ be a locally connected compact space and $\mathcal{A} \subset \uC(K)$
  a unital $\uC^*$-subalgebra. Then $\mathcal{A}$ is monotone if and only if 
  it contains the irreducible parts of all its functions.
\end{corollary}

As mentioned in \cref{section:prelim}, many abstractly defined factors in topological
dynamics, including the maximal equicontinuous factor, naturally have corresponding 
$\uC^*$-subalgebras that admit simple descriptions. Hence, \cref{prop:monotonechar} and \cref{cor_algebrachar}
provide a useful way of verifying the monotonicity of factors. For the maximal equicontinuous 
factor, this yields the following.

\begin{theorem}\label{thm:mefequi}
  Let $(K, S)$ be a topological dynamical system such 
  that $K$ is locally connected and the semigroup $S$ acts on $K$ via monotone maps.
  Then the factor map onto the maximal equicontinuous factor of $(K, S)$ is monotone.
\end{theorem}
\begin{proof}
  By \cref{example:factor} and \cref{cor_algebrachar}, it suffices to show that 
  the subalgebra
  \begin{align*}
    \mathcal{A}_{\mathrm{eq}} = \left\{f\in \uC(K) \mmid \{T_s f \mid s\in S\} \text{ is 
    equicontinuous}\right\}
  \end{align*}
  satisfies $\irr(\mathcal{A}_{\mathrm{eq}}) \subset \mathcal{A}_{\mathrm{eq}}$. So let 
  $f\in\mathcal{A}_{\mathrm{eq}}$ and $g\in \irr(f)$. Since the connectedness of a set is 
  preserved by taking preimages under monotone maps, $T_sg$ is irreducible for
  each $s\in S$
  and so $T_sg \in \irr(T_s(f))$. Therefore,
  \begin{align*}
    \{T_sg \mid s\in S\} 
    \subset \bigcup_{s\in S} \irr(T_s(f))
    = \irr\left(\{T_s(f) \mid s\in S\}\right).
  \end{align*}
  The latter set is equicontinuous by \cref{lem:irrequi}
  below and so it follows that the orbit $\{T_sg \mid s\in S\}$ of $g$ is equicontinuous as well.
  Hence, $g\in\mathcal{A}_{\mathrm{eq}}$.
\end{proof}

\begin{lemma}\label{lem:irrequi}
  Let $K$ be a locally connected compact space and $\mathcal{F} \subset \uC(K)$. 
  Then $\mathcal{F}$
  is equicontinuous if and only if $\irr(\mathcal{F})$ is.
\end{lemma}
\begin{proof}
  Suppose $\mathcal{F}$ to be equicontinuous and take $\epsilon > 0$. Then there exists 
  an entourage $V\in\mathcal{U}_K$ such that for each $f\in\mathcal{F}$ and $(x, y)\in V$  
  one has $|f(x) - f(y)| < \frac{\epsilon}{2}$. By \cref{lem:locallyconnected}, we may 
  assume that $V[x]$ is connected for each $x\in K$. Let $g\in\irr(\mathcal{F})$, i.e., 
  $g \in \irr(f_0)$ for some $f_0\in\mathcal{F}$. We claim that $|g(x) - g(y)| < \epsilon$
  for all $(x, y)\in V$ which would show that $\irr(\mathcal{F})$ is equicontinuous. 
  
  So let $(x, y)\in V$.
  If $g(x) = g(y) = 0$, it holds trivially that $|g(x) - g(y)| < \epsilon$, so assume without
  loss of generality that
  $g(x) \neq 0$. If $V[x]$ lies in $[f\neq 0]$, then it lies in the connected component of 
  $[f\neq 0]$ containing $x$ and so $y\in V[x] \subset [g\neq 0]$. Therefore, 
  \begin{align*}
    |g(x) - g(y)| = |f(x) - f(y)| < \frac{\epsilon}{2}.
  \end{align*}
  If $V[x]\not\subseteq [f\neq 0]$, there is a $z\in V[x]$ with $f(z) = 0$ and so 
  \begin{align*}
    |g(x) - g(y)| \leq |f(x)| + |f(y)| = |f(x) - f(z)| + |f(y) - f(z)| < \epsilon.
  \end{align*}
  Therefore, $\irr(\mathcal{F})$ is equicontinuous.
  The converse implication follows similarly.
\end{proof}

Of course, the most common examples of semigroups acting via monotone maps 
are given by group actions, so that 
we obtain a new proof for \cite[Theorem 2.12]{HauserJaeger2017}.

\begin{corollary}\label{cor:mefequi}
  Let $(K, G)$ be a flow on a locally connected compact space $K$. Then the factor map onto the 
  maximal equicontinuous factor is monotone.
\end{corollary}

It is known (see \cite[Theorem 3.16]{Blokh2005}) that if $(\T^2, \phi)$ is a minimal topological 
dynamical system on the two-torus, then $\phi$ is necessarily monotone, and there do exist non-invertible examples for 
such systems with a non-trivial maximal equicontinuous factor, see \cite[Theorem 3.3]{Kolyada2001}.
Therefore, there are examples in which \cref{thm:mefequi} provides meaningful information
that cannot be deduced from \cref{cor:mefequi}.

In preparation for the next section, we collect several properties of monotone subalgebras. 

\begin{proposition}\label{prop:mongen}
  Let $K$ be a locally connected compact space and 
  $\mathcal{A}\subset \uC(K)$ a unital $\uC^*$-subalgebra. 
  \begin{enumerate}
    \item\label{prop:mongen_a} If $(f_n)_{n\in\N}$ is a sequence in $\uC(K)$ 
    converging to $f\in\uC(K)$ and $f_n = h_n + g_n$ is an orthogonal 
    decomposition for each $n\in\N$, then 
    there is a subsequence $(f_{n_k})_{k\in\N}$ such that 
    $(g_{n_k})_{k\in\N}$ and $(h_{n_k})_{k\in\N}$ converge uniformly.
    If $f$ is irreducible, one of the sequences converges to 0 uniformly.
    \item\label{prop:mongen_b} If $(f_n)_{n\in\N}$ is a positive increasing sequence in 
    $\uC(K)$ converging uniformly to $f\in\uC(K)$, then for each $g\in\irr(f)$ there is a 
    positive increasing sequence $(g_n)_{n\in\N}$ such that
    $g_n\uparrow g$ and for each $n\in\N$ either $g_n = 0$ or $g_n\in\irr(f_n)$.
    \item\label{prop:mongen_c} If $f\in\mathcal{A}$, then $\irr(f)\subset\mathcal{A}$
    if and only if $\irr(|f|)\subset \mathcal{A}$.
    \item\label{prop:mongen_d} If $D\subset \mathcal{A}$ is a dense $\Q$-vector sublattice containing 
    $\1_K$,
    then $\irr(D) \subset D$ implies $\irr(\mathcal{A}) \subset \mathcal{A}$.
    \item\label{prop:mongen_e} Let $\mathcal{S}$ be a system of unital $\uC^*$-subalgebras 
    of $\uC(K)$ closed under arbitrary intersections and containing $\uC(K)$. 
    Then there exists a smallest monotone subalgebra
    $\mathcal{A}_{\mon}^{\mathcal{S}}\in\mathcal{S}$ containing $\mathcal{A}$, called 
    the \emph{monotone hull} of $\mathcal{A}$ in $\mathcal{S}$.
    \item\label{prop:mongen_f} Furthermore, suppose that
    for every separable $\uC^*$-subalgebra $\mathcal{B}\subset\uC(K)$ there is a separable 
    $\uC^*$-subalgebra $\mathcal{B}^\mathcal{S}\in\mathcal{S}$ satisfying 
    $\mathcal{B} \subset \mathcal{B}^\mathcal{S}$ and that if $(\mathcal{B}_n)_{n\in\N}$
    is an increasing sequence in $\mathcal{S}$, then $\overline{\bigcup_{n\in\N}\mathcal{B}_n} 
    \in \mathcal{S}$. Then the monotone hull $\mathcal{A}_\mon^\mathcal{S}$ is 
    separable if and only if $\mathcal{A}$ is.
  \end{enumerate}
\end{proposition}
\begin{proof}
  For \ref{prop:mongen_a}, it follows from the Arzel\`a-Ascoli theorem that the set 
  \begin{align*}
    \{g_n \mid n\in\N\} \cup \{h_n \mid n\in \N\}
  \end{align*}
  is relatively compact in $\uC(K)$ which implies the existence of the sequence $(n_k)_{n\in\N}$.
  The limits $g$ and $h$ of $(g_{n_k})_{k\in\N}$ and $(h_{n_k})_{k\in\N}$ 
  satisfy $f = g + h$ and $g\perp h$ and hence yield an orthogonal decomposition of $f$. Therefore, 
  if $f$ is irreducible, $g = 0$ or $h = 0$ which proves \ref{prop:mongen_a}.
    
  For \ref{prop:mongen_b}, one can pass to $(\1_{[g\neq 0]}f_n)_{n\in\N}$ and $\1_{[g\neq 0]}f = g$ 
  and hence assume that $f$ itself is already irreducible. Pick $x\in [f \neq 0]$. Without loss of generality,
  we may assume that $f_1(x) > 0$. Then for each $n\in\N$, there is a unique 
  $g_n \in \irr(f_n)$ with $g_n(x) = f_n(x) > 0$. 
  To see that $(g_n)_{n\in\N}$ is increasing, note that  
  for each $n\in\N$
  \begin{align*}
    [g_n \neq 0] \subset [f_n \neq 0] \subset [f_{n+1}\neq 0] = [g_{n+1}\neq 0] \cup [f_{n+1} - g_{n+1}\neq 0].
  \end{align*}
  This implies
  \begin{align*}
    [g_n\neq 0] = \left([g_n \neq 0]\cap[g_{n+1}\neq 0]\right) \cup \left([g_n \neq 0]\cap[f_{n+1}-g_{n+1}\neq 0]\right)
  \end{align*}
  and since this union is disjoint, $g_n$ is irreducible, and $x\in [g_n\neq 0]\cap[g_{n+1}\neq 0]$, 
  $[g_n\neq 0] \subset [g_{n+1}\neq 0]$. Since $(f_n)_{n\in\N}$ is increasing, 
  this yields that $(g_n)_{n\in\N}$ is increasing. Now consider the orthogonal 
  decomposition $f_n = g_n + (f_n - g_n)$. Since $g_n(x) \geq g_1(x) > 0$ for each 
  $n\in\N$, \ref{prop:mongen_a} yields that every subsequence of $(f_n - g_n)_{n\in\N}$
  has a subsequence converging to 0, showing that $f_n - g_n \to 0$. Therefore, 
  $g_n\uparrow f$.
  
  For \ref{prop:mongen_c}, note that if $f\in\mathcal{A}$, the absolute value yields a bijection 
  $\irr(f) \to \irr(|f|)$ and that $\irr(f) \subset \mathcal{A}$ hence implies 
  $\irr(|f|) \subset \mathcal{A}$.
  So suppose that, conversely, $\irr(|f|) \subset \mathcal{A}$ and let $g\in \irr(f)$. Then
  $|g|\in\irr(|f|) \subset \mathcal{A}$ and so $f|g|^{\frac{1}{n}} \in \mathcal{A}$ for each 
  $n\in\N$. Without loss of generality, we may assume that $|g| \leq 1$, in which case 
  $f|g|^{\frac{1}{n}}$ converges uniformly to $g$ as $n \to \infty$ and so
  $\irr(f)\subset\mathcal{A}$.
  
  Now let $D$ be as in \ref{prop:mongen_d} and $f\in \mathcal{A}$.
  We need to show that $\irr(f)\subset\mathcal{A}$ and by \ref{prop:mongen_c}
  we may assume that $f$ is positive. We want to use \cref{prop:mongen_b} and therefore 
  claim that $D$ contains a positive, increasing 
  sequence $(f_n)_{n\in\N}$ converging to $f$. To see this, let $(h_n)_{n\in\N}$ be a sequence in $D$ converging 
  to $f$. By passing to $(|h_n|)_{n\in\N}$, we may assume that the sequence is positive.
  Moreover, we can arrange that $\| f - h_n\| < \frac{1}{n}$ for each $n\in\N$ and 
  by passing to $g_n' = (g_n - \frac{1}{n}\1_K)_+ = \sup(g_n - \frac{1}{n}\1_K, 0)$ we may 
  therefore also assume that for each $n\in\N$
  \begin{align*}
    \left(f - \frac{2}{n}\1_K\right)_+ \leq g_n \leq \bigvee_{k=1}^n g_k \leq f.
  \end{align*}
  Since $(f - \frac{2}{n}\1_K)_+$ converges uniformly to $f$ as $n\to \infty$, 
  we can set $f_n \defeq \bigvee_{k=0}^n g_k$ for $n\in\N$ and obtain a
  sequence $(f_n)_{n\in\N}$ in $D$ such that $f_n\uparrow f$.
  Now pick $g\in \irr(f)$. Then by 
  \ref{prop:mongen_b}, there is a sequence $(g_n)_{n\in\N}$ with $g_n\uparrow g$ and 
  $g_n \in \irr(f_n)\cup\{0\}$ 
  for each $n\in\N$. Since by assumption $\irr(D) \subset D$, $(g_n)_{n\in\N}$ lies in 
  $D \subset \mathcal{A}$ and so $g\in\mathcal{A}$.
    
  \Cref{prop:mongen_e} immediately follows from \cref{cor_algebrachar} by taking the 
  intersection of all monotone $\uC^*$-subalgebras in $\mathcal{S}$ that contain $\mathcal{A}$.
  For \ref{prop:mongen_f}, it suffices to find a separable, monotone subalgebra in $\mathcal{S}$ 
  that contains $\mathcal{A}$. To this end, define increasing $\Q$-vector sublattices 
  $D_n\subset \uC(K)$ and subalgebras $\mathcal{B}_n\subset \uC(K)$ as follows: Let 
  $\mathcal{B}_0 \in\mathcal{S}$ be a separable subalgebra containing $\mathcal{A}$ and 
  let $D_0 \subset \mathcal{B}_0$ be a countable dense $\Q$-vector sublattice containing 
  $\1_K$. For $n\in\N_0$ 
  then define $\mathcal{B}_{n+1}$ to be a separable $\uC^*$-algebra in $\mathcal{S}$
  containing the $\uC^*$-algebra generated by $\irr(D_n)$ and $\mathcal{B}_n$ and let 
  $D_{n+1} \subset \mathcal{B}_n$ be a countable dense $\Q$-vector sublattice containing $D_n$
  and $\irr(D_n)$.
  Then $\mathcal{B} \defeq \overline{\bigcup_{n\in\N} \mathcal{B}_n}$ lies in $\mathcal{S}$
  and $D \defeq \bigcup_{n\in\N} D_n$ is a countable dense vector sublattice of 
  $\mathcal{B}$ satisfying $\irr(D) \subset D$. Since $\mathcal{B}$ is therefore separable, 
  contains $\mathcal{A}$, and is monotone by \ref{prop:mongen_d}, 
  $\mathcal{A}_\mon^\mathcal{S} \subset \mathcal{B}$ is separable.
\end{proof}

\begin{remark}\label{rem:mongen}
  Let $p\colon (K, G) \to (L, G)$ be a factor map of flows
  on locally connected spaces.
  Then the family $\mathcal{S}$ of $G$-invariant $\uC^*$-subalgebras of $\uC(K)$ satisfies
  the condition in \cref{prop:mongen}\ref{prop:mongen_e} and so we can consider the monotone $G$-invariant
  hull of $\mathcal{A}_L = T_p(\uC(L))$ in $\mathcal{S}$ which we denote by $T_p(\uC(L))^G_\mon$. This subalgebra
  corresponds to a monotone factor map $q\colon (K, G) \to (L_\mon, G)$ of $(K, G)$ and $p$ factorizes over 
  $q$:
  \begin{align*}
    \xymatrix{
      (K, G) \ar[r]^q \ar@/^0.9cm/[rr]^p & (L_\mon, G) \ar[r]^{\hat{p}} & (L, G)
    }
  \end{align*}
  Moreover, since $(L_\mon, G)$ corresponds to the monotone hull of $T_p(\uC(L))$, 
  it is the smallest monotone factor of $(K, S)$ over which $p$ factorizes. It is 
  not difficult to see that $(L_\mon,G)$ is therefore isomorphic to the quotient of $(K, G)$
  by the $G$-invariant equivalence relation 
  \begin{align*}
    \operatorname{Rc}(p) \defeq \left\{ (x, y)\in K\times K\mmid 
        \begin{matrix}
          \text{$x$ and $y$ are in the same } \\
          \text{connected component of } p^{-1}(p(x))
        \end{matrix}
    \right\}
  \end{align*}
  that was apparently first considered in \cite[Definition 2.2]{McMahonWu1976} and is closed 
  by \cite[Proposition 2.3]{McMahonWu1976} or the more general \cite[Proposition 2.3]{HauserJaeger2017}.
  We note for the next section that for separable $G$, 
  \cref{prop:mongen}\ref{prop:mongen_f} implies that $L_\mon$ is metrizable if 
  and only if $L$ is. Combined with \cref{thm:mefequi}, this means that if 
  $(L, G)$ is an equicontinuous metrizable factor of $(K, G)$, then so is $(L_\mon, G)$.
\end{remark}

\section{Equicontinuous factors}\label{section:mef}

Given a quotient map $p\colon X\to Y$ of topological spaces, it is generally very difficult 
to relate geometric properties of $X$ to those of $Y$. The Hahn-Mazurkiewicz theorem illustrates 
how hopeless the situation is in general: It shows that every 
non-empty, connected, locally connected, compact metric space is the quotient of the unit 
interval. Considering that this includes, in particular, all compact manifolds, it is 
clear that additional properties of $p$ are needed in order to relate the geometric structure of $X$
to that of $Y$. The following theorem shows that monotonicity is such a property.

\begin{theorem}[\mbox{\cite[Theorem 1.1]{Calcut2012}}]\label{thm:calcut}
  Let $f\colon (X, x_0) \to (Y, y_0)$ be a quotient map of pointed topological spaces, where $X$ 
  is locally path-connected and $Y$ is semi-locally simply-connected. If each fiber 
  $f^{-1}(y)$ is connected, then the induced homomorphism $f_*\colon \pi_1(X, x_0)\to \pi_1(Y, y_0)$
  of the fundamental groups is surjective.
\end{theorem}

Combining this with the previous discussion, we 
obtain our main representation result for equicontinuous factors.

\begin{theorem}\label{mthm}
  Let $(K, G)$ be a flow such that 
  $K$ is locally path-connected with finite first Betti number $\ub_1(K)$, 
  $G$ is abelian, and such that $K$ is metrizable or $G$ is separable.
  If all $G$-invariant functions $f\in\uC(K)$ are constant, then every 
  equicontinuous factor of $(K, G)$ is isomorphic to a minimal flow on some compact abelian 
  Lie group of dimension less than $\nicefrac{\ub_1(K)}{\ub_0(K)}$. More precisely,
  for every equicontinuous factor $(L, G)$ of $(K, G)$, there are
  a finite abelian group $F$ of order $|F| \leq \ub_0(K)$ and an
  $m \leq \nicefrac{\ub_1(K)}{\ub_0(K)}$ such 
  that $(L, G)$
  is isomorphic to a minimal action of $G$ on $F\times\T^m$ via rotations.
\end{theorem}
\begin{proof}
  We assume that $(M, G)$ is a monotone equicontinuous factor of $(K, G)$
  which will imply the claim for every other equicontinuous factor.
  Denote by $\theta\colon (K, G)\to (M, G)$ the corresponding factor map
  and note that $(M, G)$ is minimal since an equicontinuous 
  system is minimal if and only if every continuous $G$-invariant
  continuous function on it is constant, see \cref{rem:mtf}.
  Since $G$ is abelian and acts equicontinuously on $M$, the Ellis group 
  $\uE(M, G)$ is a compact abelian group and it is well-known that a minimal
  equicontinuous flow $(M, G)$ with an abelian group $G$ is isomorphic to 
  the minimal action $(\uE(M, G), G)$ of $G$ on the Ellis group $\uE(M, G)$ via 
  rotation (see \cite[Theorem 3.6]{Auslander1988}).
  Since $\uE(M, G) \cong M$ is the quotient of a locally connected 
  space, it follows by \cref{lem:locallyconnected} that $\uE(M, G)$
  is locally connected too.
  
  First, assume $M$ and hence $\uE(M, G)$ to be metrizable. It then follows 
  from the classification of locally connected, second-countable, compact abelian 
  groups that $\uE(M, G) \cong F\times\T^I$ for a finite group $F$ and an (at most countable)
  set $I$, see \cite[Theorem 8.34]{Hofmann2013} and \cite[Theorem 8.46]{Hofmann2013}.
  We again denote
  the induced map from $K$ to $F\times\T^I$ by $\theta$.
  Since 
  $\theta$ is monotone by \cref{thm:mefequi}, $K$ and $F\times\T^I$ have the 
  same number of connected components, and so $F$ is of order $\ub_0(K)$. Since 
  $G$ acts minimally on $F\times\T^I$ via the isomorphism $\uE(M, G) \cong F\times\T^I$,
  it follows that $G$ acts transitively on $F$ and hence on the connected components 
  of $K$. Therefore, if we fix the connected component $K_0 \defeq \theta^{-1}(\{0\}\times\T^I)$, 
  then $\ub_1(K) = \ub_0(K)\ub_1(K_0)$.
  
  Next, we show that $I$ is finite by using \cref{thm:calcut} to show that 
  $|I| \leq \nicefrac{\ub_1(K)}{\ub_0(K)}$,
  though we need to be careful since $\T^I$ is semi-locally simply connected if and only if $I$ is 
  finite. We therefore proceed by considering monotone finite-dimensional quotients: For $k\in \N$ with $k \leq |I|$, let 
  $i_1, \dots, i_k\in I$ be pairwise different and denote by $p_k\colon \T^I \to \T^k$
  the canonical projection induced by the isomorphism $\T^k \cong \T^{\{i_1,\dots,i_k\}}$.
  Moreover, let $\theta_0\colon K_0 \to \T^I$ be
  the map canonically induced by $\theta$.
  Then $p_n\circ \theta_0 \colon K_0 \to \T^k$ 
  is monotone by \cref{prop:monotonechar}\ref{part:monotonechar_b}, 
  being the composition of monotone maps. 
  Since $\T^k$ is 
  semi-locally simply connected, \cref{thm:calcut} shows that
  $p_k\circ\theta_0$ induces a surjective morphism 
  $(p_k\circ\theta_0)_*\colon \pi_1(K_0) \to \pi_1(\T^n)$. Since 
  $\pi_1(\T^n) \cong \Z^k$ is abelian, this morphism factorizes 
  through the abelianization of $\pi_1(K_0)$, which is canonically 
  isomorphic to $H_1(K_0)$ by the Hurewicz theorem. If we 
  denote by $\eta\colon H_1(K_0) \to \pi_1(\T^k)$ the induced 
  surjective group homomorphism, then 
  \begin{align*}
    k
    = \rank\left(\pi_1\left(\T^k\right)\right) 
    = \rank\left(\eta\left(H_1(K_0)\right)\right)
    \leq \rank\left(H_1(K_0)\right)
    = \ub_1(K_0)
    = \frac{\ub_1(K)}{\ub_0(K)}.
  \end{align*}
  Since $k\in\N$ was arbitrary with $k \leq |I|$, this shows that $|I| \leq \nicefrac{\ub_1(K)}{\ub_0(K)}$.
  
  Now we show that $(M,G)$ is necessarily metrizable which we only need to check for the maximal 
  equicontinuous factor $(K_\eq, G)$. By \cref{example:factor}, $K_\eq$ is metrizable if 
  and only if the
  subalgebra 
  \begin{align*}
    \mathcal{A}_{\mathrm{eq}} = \left\{ f\in \uC(K) \mmid \left\{T_\phi^n f \mmid n\in \N\right\} \text{ is equicontinuous}\right\}
  \end{align*}
  is separable.
  If $K$ is metrizable, this is the case for every $\uC^*$-subalgebra of $\uC(K)$, so assume instead 
  that $G$ is separable.
  If $\mathcal{A}_{\mathrm{eq}}$ is not separable,
  \cref{prop:mongen}\ref{prop:mongen_f} and \cref{rem:mongen} yield a sequence $(\mathcal{A}_j)_{j\in\N}$
  of strictly (!) increasing, separable, monotone, $G$-invariant $\uC^*$-subalgebras 
  of $\mathcal{A}_{\mathrm{eq}}$ which induces the following commutative diagram 
  of factor maps:
  \begin{align*}
    \xymatrix{
      & (K, G) \ar[d]_-{\theta} & & \\
      & (K_\eq, G) \ar[d]_-{\rho_{j}} \ar[dr]^-{\rho_{j-1}} &   &  \\
      \hdots \ar[r]^-{r_{j+1}} & (L_{j}, G) \ar[r]^-{r_{j}} & (L_{j-1}, G) \ar[r]^-{r_{j-1}} & \hdots 
    }
  \end{align*}
  Since each of the systems $(L_j, G)$ is metrizable and the factor maps $\rho_j\circ\theta$
  are monotone, the above discussion applies
  and we can therefore replace the diagram with 
  
  \begin{align*}
    \xymatrix{
      & (K, \phi) \ar[d]_\theta & & \\
      & (\uE(K_\eq, G), G) \ar[d]_-{\rho_{j}} \ar[dr]^-{\rho_{j-1}} &   &  \\
      \hdots \ar[r]^-{r_{j+1}} & (F_j\times\T^{m_{j}}, G) \ar[r]^-{r_{j}} & (F_{j-1}\times \T^{m_{j-1}}, G) \ar[r]^-{r_{j-1}} & \hdots 
    }
  \end{align*}
  where $m_{j}\in \{1, \dots, \nicefrac{\ub_1(K)}{\ub_0(K)}\}$ and $F_j$ is a 
  finite abelian group for each $j\in \N$.
  Since each of the systems $(F_j\times\T^{m_j}, G)$ is minimal and $G$ acts via rotations, each 
  $r_j$ is a surjective group homomorphism and so \cref{torus} below 
  implies that $m_{j-1} \leq m_{j}$.
  Moreover, since 
  $\rho_{j-1}\circ\theta$ is monotone by the choice 
  of $\mathcal{A}_{j-1}$, it follows that $\rho_{j-1}$ is 
  monotone for each $j\in\N$. But if 
  $\rho_{j-1} = r_j\circ\rho_j$ is monotone, it also follows that 
  $r_j$ is monotone for each $j\in\N$.
  
  Since $m_j \leq \nicefrac{\ub_1(K)}{\ub_0(K)}$ for each $j\in\N$, there can be only finitely many 
  $j$ such that $m_j < m_{j+1}$. In particular, there is a $J\in\N$ such that 
  $m_j = m_{j+1}$ for each $j > J$. However, if $m_j = m_{j+1}$, then $\mathcal{A}_j = \mathcal{A}_{j+1}$: 
  A surjective group homomorphism on $F\times\T^{m_{j+1}}$ must have finite kernel by 
  \cref{torus} and 
  can therefore only be monotone if its kernel is trivial, i.e., if it is an isomorphism. 
  This contradicts the strict inclusion
  $\mathcal{A}_{j} \subsetneq \mathcal{A}_{j+1}$ and shows that $\mathcal{A}_{\mathrm{eq}}$ must be separable.
  Hence, $K_\eq$ is metrizable.
  
  Now suppose $(L, G)$ is an arbitrary equicontinuous factor of $(K, G)$ and let $(K_\eq, G)$
  be the maximal equicontinuous factor of $(K, G)$. Then as in the monotone case, it follows 
  that $(L, G) \cong (\uE(L, G), G)$ and since the factor map $(K_\eq, G)\to (L, G)$
  induces a surjective group homomorphism $r\colon \uE(K_\eq, G) \to \uE(L, G)$, 
  the claim follows from the monotone case via \cref{torus}.
\end{proof}

\begin{lemma}\label{torus}
  Let $m \in \N$, $F$ be a finite abelian group, 
  and $r\colon F\times\T^m\to H$ be a continuous, surjective group homomorphism
  onto a compact group $H$.
  Then $H \cong F'\times\T^n$ for some finite abelian 
  group $F'$ of order $|F'| \leq |F|$ and $n \leq m$. Moreover,
  $n = m$ if and only if the kernel of $r$ is finite.
\end{lemma}
\begin{proof}
  If $G$ is a Lie group and $K\subset G$ is a closed normal subgroup, then $G/K$ carries 
  a canonical differentiable structure turning $G/K$ into a Lie group and
  $\pi\colon G \to G/K$ into a submersion, see \cite[Theorems 21.17 and 21.26]{Lee2012}. 
  Therefore, $H \cong F\times\T^m/\ker(r)$
  is a Lie group of dimension less than $m$. Being the quotient of $F\times\T^m$, $H$ is a 
  compact abelian Lie group and it is well-known (see \cite[Theorem 5.2]{Sepanski2007}) that this implies 
  that $H\cong F'\times\T^n$ for some finite abelian group $F'$ and $n\in \N$, proving the first statement.
  If $n = m$, 
  $r\colon F\times\T^m\to F'\times\T^n$ is a submersion between manifolds of equal dimension and 
  it thus follows from the inverse function theorem that it is in fact a local diffeomorphism. 
  Therefore, the kernel of $r$ is discrete and since $F\times\T^m$ is compact, the kernel of $r$ must 
  be finite. Conversely, if $\ker(r)$ is finite, $r$ is a local diffeomorphism and so $n = m$.
\end{proof}

Many spaces satisfy the conditions of \cref{mthm}, including compact manifolds and 
finite CW complexes for which it follows from the Seifert-van Kampen theorem that 
their first Betti number is finite.
In particular, we obtain the following 
generalization of \cite[Theorem 3.1]{HauserJaeger2017} to quotients of connected, 
simply connected Lie groups by discrete, cocompact subgroups. Important 
examples for such spaces are given by nilmanifolds, see \cite[Chapter 10]{HostKra2018}.

\begin{corollary}\label{cor:liegroupquotient}
  Let $H$ be a connected, simply connected Lie group, $\Gamma \subset H$ a 
  discrete, cocompact subgroup, $G$ an abelian group, and $(H/\Gamma, G)$ a dynamical 
  system on $H/\Gamma$ such that every $G$-invariant function $f\in \uC(H/\Gamma, G)$ is 
  constant. Then every equicontinuous factor of 
  $(H/\Gamma, G)$ is isomorphic to a minimal 
  action of $G$ on some torus $\T^m$, $m\leq \rank(\Gamma/[\Gamma, \Gamma])$, via rotations.
\end{corollary}
\begin{proof}
  The canonical map $H \to H/\Gamma$ is the universal cover of $H/\Gamma$ and it is well-known that 
  its kernel is thus isomorphic to the fundamental group of $H/\Gamma$, i.e., 
  $\pi_1(H/\Gamma) \cong \Gamma$. Since, $H$ being connected, $\ub_0(H/\Gamma) = 1$ and
  \begin{align*}
    \ub_1\left(H/\Gamma\right) 
    = \rank\left(H_1(H/\Gamma)\right) 
    = \rank\left(\pi_1(H/\Gamma)_{\mathrm{ab}}\right)
    = \rank\left(\Gamma_{\mathrm{ab}}\right)
    = \rank\left(\Gamma/[\Gamma, \Gamma]\right),
  \end{align*}
  the claim follows by \cref{mthm}.
\end{proof}

In the special case $H = \R^n$ and $\Gamma = \Z^n$, this yields the following.

\begin{corollary}\label{cor:toruscase}
  Let $(\T^n, G)$ be a flow such that $G$ is abelian and all $G$-invariant 
  functions $f\in\uC(\T^n)$ are constant. Then every equicontinuous factor of 
  $(\T^n, G)$ is isomorphic to a minimal 
  action of $G$ on some torus $\T^m$, $m\leq n$, via rotations.
\end{corollary}

Note, however, that in the case of the two-torus, \cite[Theorem 3.8]{HauserJaeger2017} is a lot
more general: They show that each monotone, minimal quotient of a 
strongly effective flow on $\T^2$ is isomorphic to a flow on $\T^2$, $\T$, or a point. They require 
neither that the acting group $G$ be abelian nor that the factor under consideration be 
equicontinuous.

If $\ub_1(K) = 0$, we also obtain the following for the maximal distal factor
of a minimal homeomorphism.

\begin{corollary}\label{cor:b1zero}
  Let $\phi\colon K\to K$ be a minimal homeomorphism on a locally path-connected space with
  $\ub_1(K) = 0$. Then the maximal distal factor of $(K, \phi)$ is trivial.
\end{corollary}
\begin{proof}
  The maximal equicontinuous factor and maximal distal factor of $(K, \phi)$ are trivial as a 
  consequence of \cref{mthm} and the Furstenberg structure theorem for distal minimal flows.
\end{proof}

In particular, there are no minimal distal transformations on such a space $K$. This includes 
simply connected manifolds such as $S^n$ for $n\geq 2$ but 
also spaces for which $\pi_1(K)$ or $H_1(K)$ are torsion groups, e.g., $\R\uP^n$ for 
$n > 1$ as 
$H_1(\R\uP^n) \cong \Z/2$ for $n > 1$. Note that a very similar result to \cref{cor:b1zero} using 
\v{C}ech cohomology exists in \cite[Theorem 3.5]{KeynesRobertson1969}.

\begin{corollary}\label{cor_spectrum}
  Let $\phi\colon K\to K$ be a homeomorphism on a locally path-connected, compact space 
  with finite first Betti number $\ub_1(K)$ and suppose that $\dim\fix(T_\phi) = 1$. Then the 
  point spectrum $\sigma_\up(T_\phi)$ of the Koopman operator $T_\phi$ on $\uC(K)$ is a 
  subgroup of $\T$ generated by at most $\ub_1(K)$ elements.
\end{corollary}
\begin{proof}
  Let $(M, \psi)$ be the maximal equicontinuous factor of $(K, \phi)$.
  By the discussion in \cref{example:factor}, the point spectrum of $T_\phi$ on $\uC(K)$ is the same 
  as that of $T_\psi$ on $\uC(M)$, so we only need to consider the system $(M, \psi)$. By \cref{mthm}, 
  $(M, \psi)$ is isomorphic to a minimal rotation $(F\times\T^n, a)$ for an abelian 
  group $F$ of order $|F| \leq \ub_0(K)$, a torus $\T^n$ of dimension $n\leq \nicefrac{\ub_1(K)}{\ub_0(K)}$,
  and an $a\in F\times\T^n$. If we denote by $L_a\colon \uC(F\times\T^n)\to \uC(F\times\T^n)$
  the Koopman operator corresponding to this rotation, then
  \begin{align*}
    \sigma_\up(L_a) = \widehat{F\times\T^n}(a) = \big\{\chi(a) \,\big|\, \chi\in \widehat{F\times\T^n}\big\}
  \end{align*}
  where $\widehat{F\times\T^n}$ denotes the Pontryagin dual of $F\times\T^n$ (see \cite[14.24]{EFHN2015}). Since
  \begin{align*}
    \widehat{F\times\T^n} 
    \cong \widehat{F}\times\widehat{\T^n} 
    \cong F\times\Z^{n},
  \end{align*}
  the claim follows from the inequalities $|F| \leq \ub_0(K)$ and $n \leq \nicefrac{\ub_1(K)}{\ub_0(K)}$.
\end{proof}

\begin{remark}
  \cref{mthm} imposes constraints on the maximal equicontinuous factor if it is 
  minimal. In the case of $K = \T^n$, these constraints are sharp: Every torus 
  $\T^m$ of dimension $m \leq n$ can be realized as the maximal equicontinuous factor 
  of an invertible system $(\T^n, \phi)$. To see this, let $(\T^m, \phi_a)$ be a
  minimal rotation with $a\in\T^m$ and let $\psi\colon \T^{n-m}\to\T^{n-m}$ be the map
  which is, after the identification $\T^{n-m} \cong [0,1)^{n-m}$, given by
  \begin{align*}
    [0,1)^{n-m} \to [0,1)^{n-m},\quad (x_1, \dots, x_{n-m}) \mapsto (x_1^2, \dots, x_{n-m}^2).
  \end{align*}  
  Then $(\T^m, \phi_a)$ is the maximal equicontinuous factor of $(\T^n, \phi_a\times\psi)$.
  
  A natural question now is whether the constraints on the maximal equicontinuous 
  factor listed in \cref{mthm} are sharp in general.
  The answer is negative: Consider 
  the wedge sum $K \defeq \bigvee_{i=1}^n\T$. Then $K$ is connected,
  locally connected, and $\ub_1(K) = n$. However, if 
  $k > 1$, there cannot be any monotone surjective map 
  $\rho \colon K \to \T^k$ since $\T^k$ is the disjoint union of uncountably
  many connected nonsingleton sets whereas $K$ is not. 
  
  In light of this example, one might look for other topological constraints on the maximal 
  equicontinuous factor and the covering dimension $\dim(K)$ of $K$ presents itself as 
  a possible candidate.
  Unfortunately, monotonicity by itself cannot yield such a bound: As observed
  by Hurewicz in \cite{Hurewicz1930}, every compact metric space embeds into a monotone image
  of $S^3$. In particular, $S^3$ has monotone quotients of infinite dimension. Therefore,
  one cannot, in general, conclude that if $p\colon K\to L$ is a monotone quotient map, that $\dim(L) \leq \dim(K)$.
  Positive results exist for factors of distal minimal flows for which this estimate does hold as shown in 
  \cite[Theorem 1.1]{Rees1977}. Without distality, results only exist in low dimensions:
  If $p\colon K\to L$ is a monotone quotient 
  map of compact spaces and $K$ is a two-manifold, then $\dim(L) \leq \dim(K)$, see \cite{Zemke1977}.
  In higher dimensions, though, one cannot hope for dimension inequalities for 
  factors without additional structural assumptions.
\end{remark}

\begin{remark}
  The commutativity of $G$ cannot be dropped in \cref{mthm} since any 
  compact Lie group acts equicontinuously on itself. However, as mentioned in the introduction,
  the maximal equicontinuous factor of a distal minimal flow on a compact manifold is 
  isomorphic to a compact abelian Lie group in the case of abelian $G$ and to a flow on a homogeneous 
  space of some compact Lie group if $G$ is nonabelian.
  One could therefore ask whether \cref{mthm}
  generalizes to nonabelian groups $G$ in an anlogous way.
  Unfortunately, the proof given above hinges on the fact that the dimension of a 
  compact abelian Lie group $H$ is precisely $\nicefrac{\ub_1(H)}{\ub_0(H)}$ and thus 
  encoded in the first two homology groups, which is false for general compact Lie groups.
\end{remark}

\end{document}